\documentclass[a4paper,10pt]{article}

\usepackage{amssymb}
\usepackage{graphicx}
\usepackage{amsthm}
\usepackage{bbm}
\usepackage{mathrsfs}
\usepackage{amsmath}

\newtheoremstyle{mytheorem}{}{}{\itshape}{}{\bfseries}{:}{\newline}{}
\newtheoremstyle{mydefinition}{}{}{}{}{\bfseries}{:}{\newline}{}
\newtheoremstyle{myproof}{}{}{}{}{\bfseries}{:}{\newline}{#1#3}

\theoremstyle{plain}
\newtheorem{thm}{Theorem}[]
\newtheorem{cor}[thm]{Corollary}
\newtheorem{lem}[thm]{Lemma}
\newtheorem{prop}[thm]{Proposition}

\theoremstyle{definition}
\newtheorem{ex}[thm]{Example}

\newtheorem*{rmk}{Remark}

\newcommand{\Qb}{\mathbb{Q}}
\newcommand{\Qt}{\tilde{\mathbb{Q}}}
\newcommand{\Pb}{\mathbb{P}}
\newcommand{\Fg}{\mathscr{F}}

\newcommand{\hs}{\hspace{2mm}}
\newcommand{\hsl}{\hspace{1mm}}
\newcommand{\ind}{\mathbbm{1}}

\author{Simon C. Harris and Matthew I. Roberts}
\title{Measure changes with extinction}

\begin{document}

\maketitle

\subsection*{Abstract}
We consider a martingale change of measure $\Qb|_{\Fg_t}:=Z_t\Pb|_{\Fg_t}$ and clarify that in general $1/Z_t$ is only a supermartingale under $\Qb$. We then give a necessary and sufficient condition for the identity $\Pb(\exists t: Z_t=0) = \Pb(Z_\infty=0)$ to hold.

\section{Introduction}
Consider two probability measures $\Pb$ and $\Qb$ on the same filtered space $(\Omega, \Fg, \Fg_t)$ along with a c\`adl\`ag adapted non-negative process $(Z_t)$ such that, for each $t$,
\[\Qb\big|_{\Fg_t} = Z_t \Pb\big|_{\Fg_t}.\]
A simple example of such a process is to take the number of particles alive at time $t$ in some branching process, and normalize it by its expected value to give $Z_t$. The process $Z$ may be in either continuous (usually $t\in\mathbb{R}_{+}$) or discrete (usually $t\in\mathbb{Z}_{+}$) time; we shall not always distinguish between the two. It is easy to see that $Z$ is necessarily a $\Pb$-martingale. We define
\[\Upsilon:=\inf\{t\geq0 : Z_t=0\};\]
we call this the extinction time of the process $Z$.

It has been claimed, in particular in Biggins \& Kyprianou \cite{biggins_kyprianou:measure_change_multitype_branching}, Englander \& Kyprianou \cite{englander_kyprianou:local_extinction_exp_growth} and Lyons \cite{lyons:simple_path_to_biggins}, that the process $1/Z_t$ is automatically a $\Qb$-martingale. This is not always true, as shown in the example below. However, in Proposition \ref{submg_prop} we show that $1/Z_t$ is a supermartingale. Since the proofs in \cite{biggins_kyprianou:measure_change_multitype_branching}, \cite{englander_kyprianou:local_extinction_exp_growth} and \cite{lyons:simple_path_to_biggins} depend only on showing that $1/Z_t$ converges $\Qb$-almost surely, the supermartingale property is sufficient and their results are unaffected.

\begin{ex}
Consider the (discrete time) Galton-Watson process with offspring distribution $L$, where
\[\Pb(L=2)=p \hs\hs\hs \hbox{ and } \hs\hs\hs \Pb(L=0)=q.\]
Let $X_n$ be the number of particles in the $n$th generation, and set
\[m=\Pb[L]=2p \hs \hbox{ and } \hs Z_n = X_n/m^n.\]
It is well-known that $Z$ is a $\Pb$-martingale. Making the change of measure to $\Qb$, we can check immediately that
\[\Qb(Z_1 = 0) = \Pb[Z_1 \ind_{\{Z_1 = 0\}}] = 0,\]
so
\begin{eqnarray*}
\Qb[1/Z_1] &=& m \sum_{j=1}^{\infty} \Qb(X_1=j)/j = m \sum_{j=1}^{\infty} \Pb[Z_1 \ind_{\{X_1=j\}}]/j\\
&=& m(2/2m)\Pb(X_1 = 2) = \Pb(X_1 = 2) = p.
\end{eqnarray*}
Since $\Qb[1/Z_0] = 1$, we see that $(1/Z_n)$ is not a $\Qb$-martingale if $p<1$.
\end{ex}

\vspace{3mm}

\noindent
In fact we show in Lemma \ref{mainlem} that in all cases
\[\Qb[1/Z_t] = \Pb(Z_t>0) = \Pb(\Upsilon>t)\]
and in Theorem \ref{inf_id} we see that the identity
\[\Qb[1/Z_\infty] = \Pb(Z_\infty>0) = \Pb(\Upsilon = \infty)\]
holds if and only if $1/Z$ is uniformly integrable. Such results, linking the extinction of the process to the event that the martingale limit is zero, are often of great value in the branching process scenario. We stress, however, that all of our results apply to general measure changes rather than just those related to branching processes.

\section{Main results}

\subsection{The $\Qb$-supermartingale property of $1/Z$}
We may easily show that, as claimed earlier, $1/Z$ is a $\Qb$-supermartingale.

\begin{prop}\label{submg_prop}
\[\Qb\left[\left.\frac{1}{Z_{t+s}} \right| \Fg_t\right] = \frac{1}{Z_t}\Pb(Z_{t+s}>0 \hspace{1mm} |\hspace{1mm} \Fg_t).\]
In particular, $(1/Z_t)$ is a $\Qb$-supermartingale.
\end{prop}

\begin{proof}
First, note that there is no extinction under $\Qb$: for all $t>0$,
\[\Qb(Z_t=0) = \Pb[Z_t \ind_{Z_t=0}] = 0.\]
Also, there is no rebirth after extinction; that is, for all $s,t>0$,
\[Z_t = 0 \Rightarrow Z_{t+s} = 0 \hs \hs \hbox{(a.s. under $\Pb$).}\]
This fact can be shown directly, using the martingale property of $Z$; however, the measure change allows us a simple proof:
\[\Pb(Z_{t+s}>0, Z_t=0) = \Pb\left[\frac{Z_{t+s}}{Z_{t+s}}\ind_{Z_{t+s}>0, Z_t=0}\right] = \Qb\left[\frac{1}{Z_{t+s}}\ind_{Z_{t+s}>0, Z_t=0}\right] = 0,\]
since $\Qb(Z_t = 0) = 0$.
Using these two facts, we see that for any $A\in\Fg_t$,
\begin{eqnarray*}
\Qb\left[\frac{1}{Z_t}\Pb(Z_{t+s}>0 | \Fg_t)\ind_A\right] &=& \Qb\left[\frac{1}{Z_t}\ind_{\{Z_t>0\}}\Pb(Z_{t+s}>0 | \Fg_t)\ind_A\right]\\
&=& \Pb\left[\frac{Z_t}{Z_t}\ind_{\{Z_t>0\}}\Pb(Z_{t+s}>0|\Fg_t)\ind_A\right]\\
&=& \Pb(Z_t>0, \hsl Z_{t+s}>0, \hsl A) \hs = \hs \Pb(Z_{t+s}>0, \hsl A)\\
&=& \Pb\left[\frac{Z_{t+s}}{Z_{t+s}}\ind_{\{Z_{t+s}>0\}}\ind_A\right] \hs = \hs \Qb\left[\frac{1}{Z_{t+s}}\ind_A\right].
\end{eqnarray*}
Thus, by definition of conditional expectation,
\[\Qb\left[\left.\frac{1}{Z_{t+s}}\right|\Fg_t\right] = \frac{1}{Z_t}\Pb(Z_{t+s}>0 | \Fg_t).\qedhere\]
\end{proof}

\begin{rmk}
Kuhlbusch \cite{kuhlbusch:weighted_bps_random_environment} gives a very similar proof of this fact, albeit in discrete time only. The proof above also has the advantage that it gives an explicit formula for the rate at which the process is decaying.
\end{rmk}

\begin{cor}
$(1/Z_t)$ is a true $\Qb$-martingale if and only if there is no extinction under $\Pb$. \label{mgextcor}
\end{cor}

\subsection{Extinction probabilities}

In work on branching processes, extinction probabilities often cause difficulties. For example, let $\Upsilon$ be the extinction time,
\[\Upsilon := \inf \{t: Z_t=0\}\]
and set
\[Z_{\infty}:=\limsup_{t\to\infty}Z_t;\]
then it can be a major problem to prove that
\begin{equation}\label{equality}
\Pb(Z_\infty>0) = \Pb(\Upsilon=\infty).
\end{equation}
In this section we give an identity that has already proved useful for this purpose (see \cite{harris_roberts:unscaled_growth}) and a necessary and sufficient condition for (\ref{equality}) to hold. We begin by stating a well-known result -- a proof can be found, for example, in Durrett \cite{durrett:prob_theory_examples} (Theorem 3.3).

\begin{lem}\label{meas_lem}
Set
\[Z_{\infty}:=\limsup_{t\to\infty}Z_t.\]
Then for $A\in\Fg$,
\[\Qb(A)=\Pb[Z_{\infty}\ind_A] + \Qb(A\cap\{Z_{\infty}=\infty\}).\]
\end{lem}

\noindent
We may now easily prove the following identity. Despite its simplicity, it can be extremely useful -- for example it is an essential ingredient in the proofs of \cite{harris_roberts:unscaled_growth}.

\begin{lem}\label{mainlem}
For any $t\in[0,\infty)$,
\[\Pb(\Upsilon>t) = \Pb(Z_t>0) = \Qb[1/Z_t];\]
also
\[\Pb(Z_\infty > 0) = \Qb[1/Z_\infty].\]
\end{lem}

\begin{proof}
Using various facts from earlier,
\[\Qb[1/Z_t] = \Qb\left[\frac{1}{Z_t} \ind_{\{Z_t>0\}}\right] = \Pb\left[Z_t\frac{1}{Z_t} \ind_{\{Z_t>0\}}\right] = \Pb(Z_t>0)\]
which establishes the first equality. For the second, we use Lemma \ref{meas_lem}. Note that
\[\Qb(Z_{\infty} = 0) = \Pb[Z_{\infty}\ind_{\{Z_{\infty}=0\}}] + \Qb(\{Z_{\infty}=0\}\cap\{Z_{\infty}=\infty\}) = 0.\]
Thus, using Lemma \ref{meas_lem} again,
\[\Qb[1/Z_{\infty}] = \Qb\left[\frac{1}{Z_{\infty}}\ind_{\{Z_{\infty}>0\}}\right] = \Pb(Z_{\infty}>0) + \Qb\left[\frac{1}{Z_{\infty}}\ind_{\{Z_{\infty}=\infty\}}\right] = \Pb(Z_{\infty}>0).\qedhere\]
\end{proof}

\noindent
This allows us to give a simple necessary and sufficient condition for (\ref{equality}) to hold.

\begin{thm}\label{inf_id}
The full identity
\[\Qb[1/Z_{\infty}] = \Pb(Z_{\infty}>0) = \Pb(\Upsilon = \infty)\]
holds if and only if the set $\{1/Z_t: t\geq0\}$ is $\Qb$-uniformly integrable.
\end{thm}

\begin{proof}
If $\{1/Z_t: t>0\}$ is $\Qb$-uniformly integrable then we have immediately that
\[\Pb(Z_{\infty}>0) = \Qb[1/Z_{\infty}] = \lim_{t\to\infty}\Qb[1/Z_t] = \lim_{t\to\infty}\Pb(\Upsilon>t) = \Pb(\Upsilon=\infty).\]
Conversely, if $\Pb(Z_{\infty}>0) = \Pb(\Upsilon=\infty)$, then as above we have
\[\Qb[1/Z_{\infty}] = \lim_{t\to\infty}\Qb[1/Z_t].\]
Thus (by Scheff\'e's lemma -- Theorem 5.10 of \cite{williams:probability_with_martingales}) $1/Z_t$ converges in $L^1$ to $1/Z_\infty$. Convergence in $L^1$ then implies uniform integrability (see Theorem 13.7 of \cite{williams:probability_with_martingales} for example); hence \mbox{$\{1/Z_t: t\geq0\}$} is $\Qb$-uniformly integrable.
\end{proof}

\section{The $\Qb$-local martingale property}
We may now ask whether $1/Z$ is even a $\Qb$-local martingale. The intuition is that if, as is often the case, $Z_t$ is some suitable rescaling of the number of particles alive at time $t$, then $1/Z_t$ is perfectly well-behaved under $\Qb$: there is always at least one particle alive, so $Z_t$ cannot get within a certain distance of zero. Thus $1/Z_t$ can only be a local martingale if it is a true martingale; but it is not a true martingale, and thus not a local martingale.

This notion is made precise in Proposition \ref{localmg} below. The result is really just a rephrasing of a standard fact about local martingales, which we state in Lemma \ref{stdmg}; we give a proof of Proposition \ref{localmg} regardless.

\begin{lem}\label{stdmg}
Suppose that $(X_t, t\geq0)$ is a local martingale. Then the following are equivalent:
\begin{itemize}
\item $X$ is a martingale;
\item For each $t>0$, $\{X_T: \hsl T \hbox{ is a stopping time, } T\leq t\}$ is uniformly integrable.
\end{itemize}
\end{lem}

\begin{prop}\label{localmg}
Suppose that extinction occurs with positive probability under $\Pb$, i.e. there exists $s>0$ such that $\Pb(Z_s=0)>0$, and that the set 
\[\{1/Z_T: \hsl T \hbox{ is a stopping time, } T\leq t\}\]
is $\Qb$-UI for each $t>0$. Then $1/Z_t$ is not a local martingale under $\Qb$.
\end{prop}

\begin{proof}
For a contradiction, suppose that $1/Z_t$ is a local martingale under $\Qb$, with a reducing sequence of stopping times $(T_n, n\geq 0)$. Then for any bounded stopping time $T\leq t$, say,
\[\Qb[1/Z_0] = \Qb[1/Z_0^{T_n}] = \Qb[1/Z_T^{T_n}] = \Qb[1/Z_{T\wedge T_n}],\]
where the second equality holds by the optional stopping theorem. Now by hypothesis $\{Z_{T\wedge T_n} : n\geq 0\}$ is UI and thus
\[\Qb[1/Z_{T\wedge T_n}] \to \Qb[1/Z_T] \hs \hbox{ as } \hs n \to \infty.\]
So $\Qb[1/Z_T] = \Qb[1/Z_0]$ for all bounded stopping times $T$, and hence by optional stopping $1/Z_t$ is a true $\Qb$-martingale. We have already shown that this is not true when there is a positive probability of extinction (Corollary \ref{mgextcor}); hence by contradiction $1/Z_t$ is not a $\Qb$-local martingale.
\end{proof}

\begin{ex}
Consider a standard branching Brownian motion, where each particle gives birth at rate $r$ to $L$ new particles with $\Pb[L]=m\in(0,\infty)$. Let $N(t)$ be the set of particles at time $t$, with particle $u$ having position $X_u(t)$. Then it is well-known that
\[Z_\lambda(t):=\sum_{u\in N_t} e^{-mrt + \lambda X_u(t) - \lambda^2 t/2}\]
is a martingale. Suppose that $\Pb(L=0)>0$. Then, making the usual change of measure to $\Qb$, we know that $1/Z_\lambda$ is not a $\Qb$-martingale. It is possible to show that it is not even a local martingale, by using the spine interpretation of the measure change -- details of this can be found in \cite{hardy_harris:new_spine_formulation} and we give only a vague explanation here. We embellish our probability space with extra information concerning one distinguished infinite line of descent, called the spine, and define a new measure $\Qt$ which is an extension of $\Qb$. Under $\Qt$ the spine moves with a drift $\lambda$, and the birth rate along the spine is also altered. The spine slmost surely survives forever under $\Qt$, and we denote its position at time $t$ by $\xi_t$. Thus almost surely under $\Qt$, for a bounded stopping time $T\leq t$ say,
\begin{eqnarray*}
\frac{1}{Z_\lambda(T)} &=& \frac{1}{\sum_{u\in N(T)} e^{-rT + \lambda X_u(T) - \lambda^2 T/2}}\\
&\leq& \frac{1}{e^{-rT + \lambda \xi_T - \lambda^2 T/2}}\\
&=& e^{(r+\lambda^2)T} \cdot e^{-\lambda(\xi_T-\lambda T) - \lambda^2 T/2}\\
&\leq& e^{(r+\lambda^2)t} \cdot e^{-\lambda(\xi_T-\lambda T) - \lambda^2 T/2}.
\end{eqnarray*}
Since $(e^{-\lambda(\xi_t-\lambda t) - \lambda^2 t/2}, t\geq 0)$ is a martingale under $\Qt$ (because $\xi$ is a Brownian motion with drift $\lambda$), by Lemma \ref{stdmg} the set
\[\{e^{-\lambda(\xi_T-\lambda T) -\lambda^2 T/2} : \hsl T \hbox{ is a stopping time, } T\leq t\}\]
is $\Qt$-uniformly integrable. Multiplying each element of the set by a constant $e^{(r+\lambda^2)t}$ does not change this property, and hence by domination
\[\{1/Z_\lambda(T): \hsl T \hbox{ is a stopping time, } T\leq t\}\]
is uniformly integrable under $\Qt$ (and so under $\Qb$). Proposition \ref{localmg} now tells us that $1/Z_\lambda$ is not a local martingale under $\Qb$.
\end{ex}

\end{document}